\documentclass{amsart}

\usepackage[utf8]{inputenc}
\usepackage[T1]{fontenc}
\usepackage{verbatim}
\usepackage{graphicx}
\usepackage{graphicx,caption2,psfrag,float,color}
\usepackage{amssymb}
\usepackage{amscd}
\usepackage{amsmath}

\usepackage[T1]{fontenc}

\newtheorem{theorem}{Theorem}[section]

\newtheorem{lemma}[theorem]{Lemma}

\numberwithin{equation}{subsection}
\newtheorem{definition}[theorem]{Definition}

\title{Asymptotics for moments of certain cotangent sums}
\author{Helmut Maier and Michael Th. Rassias}
\date{\today}
\address{Department of Mathematics, University of Ulm, Helmholtzstrasse 18, 89081 Ulm, Germany.}
\email{helmut.maier@uni-ulm.de}
\address{Institute of Mathematics, University of Zurich, CH-8057, Zurich, Switzerland
 \& Institute for Advanced Study, Program in Interdisciplinary Studies,
1 Einstein Dr, Princeton, NJ 08540, USA.}
\email{michail.rassias@math.uzh.ch, michailrassias@math.princeton.edu}
\thanks{}

\begin{document}

 \maketitle
 
\begin{abstract} 
In this paper we improve a result on the order of magnitude of certain cotangent sums 
associated to the Estermann and the Riemann zeta functions.
%\textbf{Key words:} Cotangent sums; equidistribution; moments; continued fractions; measure.\\
%\textbf{2000 Mathematics Subject Classification:}  26A12; 11L03; 11M06.%
%\newline

\end{abstract}
\section{Introduction}
The authors in joint work \cite{mr} and the second author in his thesis \cite{rasthesis}, investigated 
the distribution of cotangent sums
$$c_0\left(\frac{r}{b}\right)=-\sum_{m=1}^{b-1}\frac{m}{b}\cot\left(\frac{\pi m r}{b}\right)$$
as $r$ ranges over the set
$$\{r\::\: (r, b)=1,\ A_0b\leq r\leq A_1 b\}\:,$$
where $A_0$, $A_1$ are fixed with $1/2<A_0<A_1<1$ and $b$ tends to infinity.\\
They could show that 
$$H_k=\int_0^1\left(\frac{g(x)}{\pi}\right)^{2k}dx\:,$$
where
$$g(x)=\sum_{l\geq 1}\frac{1-2\{lx\}}{l}\:,$$
a function that has been investigated by de la Bret\`eche and Tenenbaum \cite{bre},
as well as Balazard and Martin \cite{balaz1, balaz2}. Bettin \cite{bettin} could replace the interval $(1/2, 1)$ for $A_0$, $A_1$ by the interval $(0,1)$.\\
In \cite{mr2}, Theorem 1.1 the authors could determine the order of magnitude of $H_k$. There are constants $c_1$, $c_2>0$, such that 
\[
c_1\Gamma(2k+1)\leq \int_0^1 g(x)^{2k}dx\leq c_2\Gamma(2k+1)\:,\tag{1.1}
\]
for all $k\in \mathbb{N}$, where $\Gamma(\cdot)$ stands for the Gamma function.\\
%Let
%\[
%A=\int_0^\infty \frac{\{t\}^2}{t^2}dt\:.\tag{1.2}
%\]
In this paper we extend the result of (1.1) to an asymptotic formula valid for arbitrary natural exponents.
\begin{theorem}\label{main}
Let 
$$A=\int_0^\infty \frac{\{t\}^2}{t^2}dt$$
and $K\in\mathbb{N}$. There is an absolute constant $C>0$, such that
$$\int_0^1|g(x)|^Kdx=2e^{-A}\Gamma(K+1)(1+O(\exp(-CK))),$$
for $K\rightarrow \infty$.
\end{theorem} 
\section{Overview and preliminary results}
Like in the proof of (1.1), a crucial role is played by the relation of $g(x)$ to 
Wilton's function, established by Balazard and Martin \cite{balaz2} and results about
operators related to continued fraction expansions due to Marmi, Moussa and Yoccoz \cite{Marmi}.\\
We recall some fundamental definitions and results.
\begin{definition}\label{def21}
Let $X=(0,1)\setminus\mathbb{Q}$. Let $\alpha(x)=\{1/x\}$ for $x\in X$. The iterates
$\alpha_k$ of $\alpha$ are defined by $\alpha_0(x)=x$ and 
$$\alpha_k(x)=\alpha(\alpha_{k-1}(x)),\ \text{for}\ k>1.$$
\end{definition}
\begin{lemma}\label{lem22}
Let $x\in X$ and let
$$x=[a_0(x); a_1(x),\ldots, a_k(x),\ldots]$$
be the continued fraction expansion of $x$. We define the partial quotient of 
$p_k(x)$, $q_k(x)$:
$$\frac{p_k(x)}{q_k(x)}=[a_0(x); a_1(x),\ldots, a_k(x)],\ \text{where},\ (p_k(x), q_k(x))=1\:.$$
Then we have 
$$a_k(x)=\left\lfloor \frac{1}{\alpha_{k-1}(x)}\right\rfloor\:,$$
$$p_{k+1}=a_{k+1}p_k+p_{k-1}$$
and
$$q_{k+1}=a_{k+1}q_k+q_{k-1}\:.$$
\end{lemma}
\begin{proof}
This is Lemma \ref{lem22} of \cite{mr2}.
\end{proof}
\begin{definition}\label{def23}
Let $x\in X$. Let also
$$\beta_k(x)=\alpha_0(x)\alpha_1(x)\cdots \alpha_k(x),\ \beta_{-1}(x)=+1$$
$$\gamma_k(x)=\beta_{k-1}(x)\log\frac{1}{\alpha_k(x)},\ \text{where}\ k\geq 0,$$
so that $\gamma_0(x)=\log(1/x)$.\\
The number $x$ is called a \textbf{Wilton number} if the series
$$\sum_{k\geq 0}(-1)^k\gamma_k(x)$$
converges.\\
Wilton's function $\mathcal{W}(x)$ is defined by
$$\mathcal{W}(x)=\sum_{k\geq 0}(-1)^k\gamma_k(x)$$
for each Wilton number $x\in (0,1)$.
\end{definition}
\begin{lemma}\label{lem24}
A number $x\in X$ is a Wilton number if and only if $\alpha(x)$ is a 
Wilton number. In this case we have:
$$\mathcal{W}(x)=\log\frac{1}{x}-x\mathcal{W}(\alpha(x)).$$
\end{lemma}
\begin{proof}
This is Lemma 2.4 of \cite{mr2}.
\end{proof}
\begin{definition}\label{def25}
Let $p>1$ and $T\::\: L^p\rightarrow L^p$ be defined by 
$$Tf(x)=xf(\alpha(x)).$$
The measure $m$ is defined by 
$$m(\mathcal{E})=\frac{1}{\log 2}\int_{\mathcal{E}}\frac{dx}{1+x},$$
where $f$ is any measurable subset of $(0,1)$.
\end{definition}
\begin{lemma}\label{lem26}
Let $p>1$, $n\in \mathbb{N}$.\\
(i) The measure $m$ is invariant with respect to the map $\alpha$, i.e.
$$m(\alpha(\mathcal{E}))=m(\mathcal{E})\:,$$
for all measurable subsets of $\mathcal{E}\subset (0,1)$.\\
(ii) For $f\in L^p$ we have 
$$\int_0^1|T^nf(x)|^p dm(x)\leq g^{(n-1)p}\int_0^1|f(x)|^pdm(x),$$
where 
$$g:=\frac{\sqrt{5}-1}{2}<1.$$
\end{lemma}
\begin{proof}
This is Lemma 2.8 of \cite{mr2}.
\end{proof}
\begin{lemma}\label{lem27}
There is a bounded function $H\::\: (0,1)\rightarrow \mathbb{R}$, which
is continuous in every irrational number, such that 
$$g(x)=\mathcal{W}(x)+H(x)\:.$$
\end{lemma}
\begin{proof}
See Lemma 2.5 of \cite{mr2}.
\end{proof}
Lemma 2.5 of \cite{mr2} is based on \cite{balaz2}. In the proof of (1.1) we 
only use the boundedness of $H$.\\
The key to the improvement of (1.1) is the use of more subtle properties
of $H$. We recall the following definitions and results from \cite{balaz2}.
\begin{definition}\label{def28}
For $\lambda\geq 0$, we set
\begin{align*}
&A(\lambda):=\int_0^\infty\{t\}\{\lambda t\}\frac{dt}{t^2}\:,\\
&F(x):=\frac{x+1}{2}A(1)-A(x)-\frac{x}{2}\log x\:,\\
&G(x):=\sum_{j\geq 0}(-1)^j\beta_{j-1}(x)F(\alpha_j(x))\:,\\
&B_1(t):=t-\lfloor t\rfloor -1/2,\ \text{the first Bernoulli function}\:,\\
&B_2(t):=\{t\}^2-\{t\}+1/6,\ (t\in\mathbb{R}) \ \text{the second Bernoulli function}\:.
\end{align*}
For $\lambda\in\mathbb{R}$, let 
$$\phi_2(\lambda):=\sum_{n\geq 1}\frac{B_2(n\lambda)}{n^2}\:.$$
\end{definition}
\begin{lemma}\label{lem29}
It holds
$$A(\lambda)=\frac{\lambda}{2}\log\frac{1}{\lambda}+\frac{1+A(1)}{2}\:\lambda+
O(\lambda^2),\ \ \text{as}\ \lambda\rightarrow 0\:.$$
\end{lemma}
\begin{proof}
By \cite{balaz2}, Proposition 31, formula (74), we have:
$$A(\lambda)=\frac{\lambda}{2}\log\frac{1}{\lambda}+\frac{1+A(1)}{2}\:\lambda+\frac{\lambda^2}{2}\phi_2\left(\frac{1}{\lambda}\right)-\int_{1/\lambda}^\infty\phi_2(t)\frac{dt}{t^3}\:.$$
From Definition \ref{def28}, it follows that $\phi_2(t)$ is bounded. 
Therefore
$$\frac{\lambda^2}{2}\phi_2\left(\frac{1}{\lambda}\right)=O(\lambda^2)$$
and
$$\int_{1/\lambda}^\infty\phi_2(t)\frac{dt}{t^3}=O(\lambda^2).$$
\end{proof}
\begin{lemma}\label{lem210}
We have
$$H(x)=2\sum_{j\geq 0}(-1)^{j-1}\beta_{j-1}(x)F(\alpha_j(x)).$$
\end{lemma}
\begin{proof}
In \cite{balaz2} the function $\Phi_1$ is defined by
\[
\Phi_1(t):=\sum_{n\geq 1}\frac{B_1(nt)}{n}=\sum_{n\geq 1}\frac{\{nt\}-1/2}{n}\:.\tag{2.1}
\]
Thus we have
\[
g(x)=-2\Phi_1(x)\:.\tag{2.2}
\]
By Proposition (2) of \cite{balaz2} we obtain
\[
\Phi_1(x)=-\frac{1}{2}\mathcal{W}(x)+G(x)\tag{2.3}
\]
almost everywhere.\\
The proof of Lemma \ref{lem210} follows now from Lemma \ref{lem27},
(2.1), (2.2) and (2.3) by the choice 
\[
H=-2G\:.\tag{2.4}
\]
\end{proof}
\section{Proof of Theorem \ref{main}}
\begin{definition}\label{def31}
Let $d, h\in \mathbb{N}_0$, $h\geq 1$, $u, v\in(0,\infty)$. Then we 
define
$$\mathcal{J}(d, h, u, v):=\{x\in X\::\: T^dl(x)\geq u\ \text{and}\ T^{d+h}l(x)\geq v\}\:.$$
\end{definition}
\begin{lemma}\label{lem32}
We have
$$m(\mathcal{J}(d, h, u, v))\leq 2\exp\left(-2^{\frac{h-2}{2}}v\exp\left(2^{\frac{d-2}{2}}u\right)\right)$$
\end{lemma}
\begin{proof}
This is Lemma 2.13 of \cite{mr2}.
\end{proof}
\begin{definition}\label{def33}
For $n\in\mathbb{N}$, $x\in X$, we define
$$\mathcal{L}(x,n):=\sum_{v=0}^n(-1)^v(T^vl)(x)\:,$$
where $l(x)=\log(1/x)$.
\end{definition}
\begin{definition}\label{def34}(Definition 2.14 of \cite{mr2})\\
We set $j_0:=L-\left\lfloor\frac{L}{100}\right\rfloor$, $C_2:=1/400.$ For 
$j\in \mathbb{Z}$, $j\leq j_0$, we define the intervals
$$I(L, j):=(\exp(-L+j-1),\: \exp(-L+j))\:.$$
For $v\in\mathbb{N}_0$, we set
$$a(L, v):=\exp(-C_2L+v).$$
$$\mathcal{T}(L, j, 0):=\{x\in I(L, j)\cap X\::\: |\mathcal{L}(x, n)-l(x)|\leq \exp(-C_2L)\}\:,$$
and for $v\in\mathbb{N}$, we set
$$\mathcal{T}(L, j, v):=\{x\in I(L, j)\cap X\::\: a(L, v-1)\leq |\mathcal{L}(x, n)-l(x)|\leq a(L, v)\}\:.$$
For $v, h\in\mathbb{Z}$, $v\geq 1$, $h\geq 0$, we set
$$U(L, j, v, h):=\{x\in \mathcal{T}(L, j, v)\::\: T^hl(x)\geq 2^{-h}a(L, v-1)\}\:.$$
\end{definition}
\begin{lemma}\label{lem35}
There are constants $C_3, C_4>0$, such that for $v\geq 1$, we have
$$m(\mathcal{T}(L, j, v))\leq C_3\exp\left(-C_4\exp\left(-C_2L+v-1+\frac{1}{2}(L-j)\right)\right)\:.$$
\end{lemma}
\begin{proof}
This is lemma 2.15 of \cite{mr2}.
\end{proof}
\begin{definition}\label{def36}(Definition 2.16 of \cite{mr2})\\
We set
$$x_0:=\exp\left(-\left\lfloor\frac{L}{100}\right\rfloor\right)\:.$$
\end{definition}
\begin{lemma}\label{lem36}
Let $L\in\mathbb{N}$, then\\
(i) $$\int_0^1 l(x)^Ldx=\Gamma(L+1)$$
(ii) There is a constant $C_5>0$, such that
$$\int_{x_0}^1 l(x)^Ldx=O(\Gamma(L+1)\exp(-C_5L))\:.$$
\end{lemma}
\begin{proof}
This is parts (i) and (ii) of Lemma 2.17 of \cite{mr2}.
\end{proof}
\begin{lemma}\label{lem37}
Let $1<p\leq 2$, such that $pL\in\mathbb{N}$. There is $n_0\in\mathbb{N}$ and a constant $C_6>0$, such that for $n\geq n_0$, we have:
$$\int_0^{x_0}|\mathcal{L}(x,n)^L-l(x)^L|^p dm(x)\leq \Gamma(pL+1)\exp(-C_6L)\:.$$
\end{lemma}
\begin{proof}
We write
\[
\mathcal{L}(x, n):=l(x)(1+R(x, n))\:.\tag{3.1}
\]
Let $j\leq j_0$. Then by Definition \ref{def34}, for 
$x\in \mathcal{T}(L, j, v)$ we have $l(x)\geq L-j$ and therefore we get
\[
l(x)\geq \frac{L}{200}\:.\tag{3.2}
\]
By Definition \ref{def34} we also have
\[
|\mathcal{L}(x, n)-l(x)|\leq \exp(-C_2L+v)\:.\tag{3.3}
\]
From (3.2) and (3.3) we have:
\[
|R(x, n)|\leq \frac{200}{L}\exp(-C_2L+v)\:.\tag{3.4}
\]
We distinguish two cases:\\
\textit{Case 1:} Let $v=0$.\\ 
From (3.4) we have 
\[
|R(x, n)|\leq \exp\left(-\frac{C_2}{2}L\right)\:.\tag{3.5}
\]
\[
\int_{\mathcal{T}(L,j,0)}|\mathcal{L}(x,n)^L-l(x)^L|^pdx\leq \int_{\mathcal{T}(L, j, 0)}l(x)^{pL}|(1+R(x,n))^L-1|^pdx\tag{3.6}
\]
From (3.5) and (3.6) we have:
\[
\int_{\mathcal{T}(L,j,0)}|\mathcal{L}(x,n)^L-l(x)^L|^pdx\leq \exp\left(-\frac{C_2}{3}L\right)\int_{\mathcal{T}(L,j,0)}l(x)^{pL}dx\:. \tag{3.7}
\]
\textit{Case 2:} Let $v\geq1$.\\
Because of the fact that 
$$L-j\geq \frac{L}{100}\:,$$
we have for an appropriate constant $C_7>0$ that
$$\max_{x\in I(L,j)}l(x)^L\leq C_7\min_{x\in I(L,j)}l(x)^L$$
and therefore from (3.4), it follows that 
\begin{align*}
&\int_{\mathcal{T}(L,j,v)}|\mathcal{L}(x,n)^L-l(x)^L|^pdx\leq \exp(-C_2L+v)m(\mathcal{T}(L,j,v))\max_{x\in I(L,j)}l(x)^{pL}\tag{3.8}\\
&\leq C_3C_7\exp\left(-C_4\exp\left(-C_2L+v-1+\frac{1}{2}(L-j)\right)\right)\exp(-C_2L+v)\min_{x\in I(L,j)}l(x)^{pL}\:.
\end{align*}
From (3.7) and (3.8), we obtain for $j\leq j_0$, the following
\[
\int_{I(L,j)\cap X}|\mathcal{L}(x,n)^L-l(x)^L|^pdx\leq \exp\left(-\frac{C_2}{3}L\right)\int_{I(L,j)}l(x)^{pL}dx\tag{3.9}
\]
The result of Lemma \ref{lem37} now follows from Lemma \ref{lem36}
by summing (3.9) for \mbox{$j\leq j_0$.}
\end{proof}
\begin{lemma}\label{lem38}
Let $1<p\leq2$ and $pL\in\mathbb{N}$. There is a constant $C_8>0$, such that
$$\int_{x_0}^{1/2}|\mathcal{L}(x,n)|^{pL}dx\leq \Gamma(pL+1)\exp(-C_8L)\:.$$
\end{lemma}
\begin{proof}
Lemma \ref{lem38} follows if we apply Lemma 2.22 from \cite{mr2}
with $pL$ instead of $L$.
\end{proof}
\begin{lemma}\label{lem39}
Let $0<\alpha<1$. Then, there is a constant $C=C(\alpha)>0$, such that
$$\int_0^{1/2}x^\alpha l(x)^Ldx\leq \Gamma(L+1)\exp(-CL)\:,$$
for all $L\in\mathbb{N}$.
\end{lemma}
\begin{proof}
We have
$$\int_0^{1/2}x^\alpha l(x)^Ldx\leq \sum_{0\leq j\leq j_0}\int_{I(L,j)}x^\alpha l(x)^L dx+\int_{x_0}^{1/2}x^\alpha l(x)^Ldx\:.$$
For $x\in I(L,j)=(\exp(-L+j-1),\:\exp(-L+j))$ we have $l(x)\leq L-(j-1)$ and
therefore
$$l(x)^L=O(L^Le^{-j}).$$
Therefore, by Stirling's formula
\begin{align*}
\int_{I(L,j)}x^\alpha l(x)^Ldx&=O(L^L\exp((\alpha+1)(-L+j)-j)\\
&=O(\Gamma(L+1)\exp(-\alpha L+(\alpha-1)j+\epsilon L),
\end{align*}
for all $\epsilon>0$, which proves Lemma \ref{lem39}.
\end{proof}
\begin{lemma}\label{lem310}
For $m\in\mathbb{N}_0$, $x\in X$, we have
$$\alpha_m(x)\alpha_{m+1}(x)\leq \frac{1}{2}\:.$$
\end{lemma}
\begin{proof}
This is Lemma 2.11 of \cite{mr2}.
\end{proof}
\begin{definition}\label{def311}
For $l_1, l_2\in\mathbb{N}_0$, $0\leq l_1+l_2\leq K$, we set
$$\int_{(l_1,l_2)}:=\int_0^{1/2}\mathcal{L}(x,n)^{K-l_1-l_2}H(x)^{l_1}
((-1)^{n+1}T^{n+1}\mathcal{W}(x))^{l_2}dx\:.$$
\end{definition}
\begin{lemma}\label{insert3}
There is a constant $C_9>0$, such that
$$\int_0^{1/2}|g(x)^K-|g(x)|^K| dx\leq \Gamma(K+1)\exp(-C_9K)\:.$$
\end{lemma}
\begin{proof}
Let 
$$x\in I(K,j)=(\exp(-K+j-1, \exp(-K+j))\:.$$
Let
$$\mathcal{Y}(K,j)=\{x\in I(K,j)\::\: g(x)\leq 0\}\:.$$
For $x\in \mathcal{Y}(K,j)$ we must have
\[
x\in \mathcal{T}(K,j,v)\ \ \text{for}\ v\geq C_2K\ \text{or}\ |T^n\mathcal{W}(x)|\geq K-j-H\:, \tag{3.10}
\]
where 
$$H=\sup_{x\in[0,1]}|H(x)|\:.$$
For $w\in\mathbb{N}$, let
\[
\mathcal{V}(K,j,w,n)=\{x\in I(L,j)\::\: L-j-H+w\leq |T^n\mathcal{W}(x)|\leq L-j-H+w+1\}\:. \tag{3.11}
\]
Let
\[
\mathcal{Z}(K,j,w,n)=\mathcal{T}(K,j,v)\cap \mathcal{V}(K,j,w,n)\:.\tag{3.12}
\]
By Lemma \ref{lem26} (ii) we have:
$$m(\mathcal{V}(K,j,w,n))(K-j-H+w)^2\leq \int_{\mathcal{V}(L,j,w)}|T^n\mathcal{W}(x)|^2dm(x)\leq g^{2(n-1)}\int_0^1|\mathcal{W}(x)|^2dm(x)\:.$$
Thus
\[
m(\mathcal{V}(K,j,w,n))\leq g^{2(n-1)}\int_0^1|\mathcal{W}(x)|^2dm(x)\: (L-j-H+w)^{-2}\:.\tag{3.13}
\]
We have 
\[
|g(x)^K-|g(x)|^K|\leq 2|g(x)|^K\tag{3.14}
\]
and for $x\in\mathcal{Z}(K,j,w,n)$
\[
|g(x)|\leq b(x,K,j,n)+|\mathcal{L}(x,n)-l(x)|,\tag{3.15}
\]
where $b(x,K,j,n):=l(x)+L-j+w+1$. Thus, from (3.14) we get
\begin{align*}
\int_{\mathcal{Z}(K,j,w,n)}|g(x)^K-|g(x)|^K|dx\leq&2^K\left(\sup_{x\in I(K,j)}|b(x,K,j,n|^K+\int_{I(K,j)}|\mathcal{L}(x,n)-l(x)|^Kdx\right)\tag{3.16}\\
&\times\left(m(\mathcal{T}(K,j,v))+m(\mathcal{V}(K,j,w,n))\right)\:.
\end{align*}
From Lemma \ref{lem35}, Lemma \ref{lem37}, (3.15), (3.16) we get by summation over 
$j, v$ and $w$:
\[
\int_0^{x_0}|g(x)^K-|g(x)|^K|dx\leq \Gamma(K+1)\exp(-C_{10}K),\tag{3.17}
\]
where $x_0:=x_0(K)=\exp(- \left\lfloor \frac{K}{100}\right\rfloor)$. From Lemma \ref{lem35}, we obtain:
\[
\int_{x_0}^{1/2}|g(x)^K-|g(x)|^K|dx\leq \Gamma(K+1)\exp(-C_{11}K)\:.\tag{3.18}
\]
Lemma \ref{insert3} now follows from (3.17) and (3.18).
\end{proof}
\begin{lemma}\label{lem312}
We have
$$\int_0^{1/2}g(x)^{K}dx=\sum_{\substack{(l_1,l_2)\in\mathbb{N}_0^2\\ 0\leq l_1+l_2\leq K}}\frac{K!}{(K-l_1-l_2)!\:l_1!\:l_2!}\int_{(l_1,l_2)}\:.$$
\end{lemma}
\begin{proof}
From formula (3) of \cite{mr2} we have:
$$\mathcal{W}(x)=\mathcal{L}(x,n)+(-1)^{n+1}T^{n+1}\mathcal{W}(x)\:.$$
By Lemma \ref{lem27}, we obtain
$$g(x)=\mathcal{L}(x,n)+H(x)+(-1)^{n+1}T^{n+1}\mathcal{W}(x)\:.$$
Lemma \ref{lem312} now follows by the Multinomial Theorem.
\end{proof}
\begin{definition}\label{def313}
For $(l_1,l_2)$ as in Definition \ref{def311} we set
$$\int_{(l_1,l_2)}^{(1)}:=\int_0^{1/2}l(x)^{K-l_1-l_2}H(x)^{l_1}[(-1)^{n+1}T^{n+1}l(x)]^{l_2}dx$$
$$\int_{(l_1,l_2)}^{(2)}:=\int_0^{1/2}(\mathcal{L}(x,n)^{K-l_1-l_2}-l(x)^{K-l_1-l_2})H(x)^{l_1}[(-1)^{n+1}T^{n+1}\mathcal{W}(x)]^{l_2}dx\:.$$
\end{definition}
\begin{lemma}\label{lem3155}
$$\int_{(l_1, l_2)}=\int_{(l_1, l_2)}^{(1)}+\int_{(l_1, l_2)}^{(2)}\:.$$
\end{lemma}
\begin{proof}
Obvious.
\end{proof}
We now show, that the integrals $\int_{(l_1, l_2)}^{(2)}$ for all $l_1, l_2$ and
$\int_{(l_1, l_2)}^{(1)}$, if $l_2>0$ are negligible.
\begin{lemma}\label{lem315}
There is an $n_0=n_0(K)\in\mathbb{N}$, such that for $n\geq n_0$
we have for $i=1, 2$ and all $l_1\leq K$ and $l_2>0$ the following
$$\int_{(l_1, l_2)}^{(i)}\leq (K(2K)!)^{-1}\:.$$
\end{lemma}
\begin{proof}
We choose $1<p\leq 2$. We set $L=K-l_1-l_2$ and apply Lemma \ref{lem37} with $p=2$ to obtain from the inequality of Cauchy-Schwarz:
$$\int_{(l_1, l_2)}^{(i)}\leq \left(\int_0^{1/2}I(x)^2dx\right)^{1/2}\left(\int_0^{1/2}|T^{n+1}\mathcal{W}(x)|^{2l_2}dx\right)^{1/2}\sup_{x\in[0,1/2)}|H(x)^{l_2}|  \:,$$
where 
$$I(x):=l(x)^L,\ \text{for}\ i=1$$ 
and 
$$I(x):=\mathcal{L}(x,n)^L-l(x)^L,\ \ \text{for}\ i=2.$$
By Lemma \ref{lem26} we obtain the result if we choose $n_0$ sufficiently large.
\end{proof}
\begin{lemma}\label{lem315}
Assume $L_0$ is sufficiently large and that $L:=K-l_1\geq L_0$.
There are constants $C_9, C_{10}>0$, such that 
$$\left|\int_{(l_1, 0)}^{(2)}\right|\leq C_0^{\:l_1}K^{l_1}\Gamma(K+1-l_1)\exp(-C_{13}K)\:.$$
\end{lemma}
\begin{proof}
Let $|H(x)|\leq C_{11}$ with $C_{11}>0$. We choose $p$, $1<p\leq 2$, such that $pL\in\mathbb{N}$.\\
We define $\epsilon>0$ by $(1-\epsilon)^{-1}=p$. Then by Lemma \ref{lem37} and H\"older's inequality we have
\begin{align*}
\int_{(l_1,0)}^{(2)}&\leq \left(\int_0^{1/2}|\mathcal{L}(x,n)^L-l(x)^L|^pdx\right)^{1/p}\left(\int_0^{1/2}|H(x)|^{l_1/\epsilon}dx\right)^{\epsilon}\\
&\leq \Gamma(pL+1)^{1/p}\exp\left(-\frac{C_6}{p}L\right)C_{11}^{\:l_1}\:.
\end{align*}
By Stirling's formula
$$\int_{(l_1,l_2)}^{(2)}\leq (pL)^L\exp\left(-\frac{L-3\epsilon}{p}\right)\exp\left(-\frac{C_6}{p}L\right)$$
for sufficiently large $L$.\\
Since $\epsilon \rightarrow 0$ for $L\rightarrow \infty$, the result of 
Lemma \ref{lem315} follows.
\end{proof}
\begin{lemma}\label{lem317}
There is a constant $C_{15}>0$, such that
$$\int_0^{1/2} g(x)^{K}dx=\sum_{0\leq l_1\leq K}\binom{K}{l_1}\int_{(l_1,0)}^{(1)}+O(\Gamma(K+1)\exp(-C_{15}K))\:.$$
\end{lemma}
\begin{proof}
This follows from Lemms \ref{lem3155} - \ref{lem315}.
\end{proof}
\begin{definition}\label{def318}
Let $0\leq m\leq l_1$. Then we set
$$\int^{(l_1,m)}:=\int_0^{1/2}l(x)^{2k-l_1}(-2F(x))^{l_1-m}\left(\sum_{j>0}(-1)^{j-1}\beta_{j-1}F(\alpha_j(x))\right)^mdx\:.$$
\end{definition}
\begin{lemma}\label{lem319}
$$\int_{(l_1,0)}^{(1)}=\sum_{m=0}^{l_1}\binom{l_1}{m}\int^{(l_1,m)}\:.$$
\end{lemma}
\begin{proof}
This follows from Lemma \ref{lem310}, Definition \ref{def313}, \ref{def318} and the Binomial Theorem.
\end{proof}
\begin{lemma}\label{lem320}
There is a constant $C_{13}>0$, such that
$$\int_0^{1/2}g(x)^{K}dx=\sum_{0\leq l_1\leq K}\binom{K}{l_1}\int^{(l_1,0)}+O\left(\Gamma(K+1)\exp(-C_{13}K)\right)\:.$$
\end{lemma}
\begin{proof}
Let $m>0$. We have 
$$\beta_{j-1}=x\alpha_1(x)\cdots \alpha_{j-1}(x)\:.$$
By Lemma \ref{lem310} we have for an absolute constant $C_{14}$ the following
$$\left|\sum_{j>0}(-1)^{j-1}\beta_{j-1}F(\alpha_j(x))\right|<C_{14}x,\ \text{if}\ x\in(0,1)\:.$$
We also have 
$$|-2F(x)|<C_{15}\:.$$
By Lemma \ref{lem39}, we therefore have
\begin{align*}
\left|\int^{(l_1,m)}\right|&\leq C_{15}^{\:l_1-m} \left|\int_0^{1/2}l(x)^{K-l_1}\left(\sum_{j>0}(-1)^{j-1}\beta_{j-1}F(\alpha_j(x))\right)^m\right|dx\:\\
&\leq \Gamma(K-l_1+1)(3C_{14}C_{15})^{l_1}\:.
\end{align*}
Lemma \ref{lem320} follows by summation over $l_1$.
\end{proof}
\begin{definition}\label{def321}
For $0\leq l_1\leq K$ we set
$${Int}(l_1):=\int_0^{1/2}l(x)^{K-l_1}(-2F(x))^{l_1}dx\:.$$
For $0\leq m\leq l_1$ we set
$$Int(l_1,m):=\int_0^{1/2}l(x)^{K-l_1}(-A(1))^{l_1-m}R(x)^mdx\:,$$
where 
$$R(x):=-xA(1)+A(x)+\frac{x}{2}\log x\:.$$
\end{definition}
\begin{lemma}\label{lem322}
We have
$$Int(l_1)=\sum_{m=0}^{l_1}\binom{l_1}{m}Int(l_1,m)\:.$$
\end{lemma}
\begin{proof}
This follows by Definition \ref{def321} and the Binomial Theorem.
\end{proof}
\begin{lemma}\label{lem323}
There is a constant $C_{16}>0$, such that
$$\int_0^{1/2}g(x)^{K}dx=\sum_{0\leq l_1\leq K}\binom{K}{l_1}\int_0^{1/2}l(x)^{K-l_1}(-A(1))^{l_1}dx+O(\Gamma(K+1)\exp(-C_{18}K))\:.$$
\end{lemma}
\begin{proof}
This follows in a similar manner as the result of Lemma \ref{lem320}
by application of Lemma \ref{lem39} and summation over $l_1$.
\end{proof}
\section{Conclusion of the proof of Theorem \ref{main}}
We have
$$\binom{K}{l_1}\int_0^1l(x)^{K-l_1}dx=\binom{K}{l_1}\Gamma(K-l_1+1)=\frac{1}{l_1!}\Gamma(K+1)\:.$$
From Lemmas \ref{lem36} and \ref{lem323}, we therefore get
\[
\int_0^{1/2}g(x)^{K}dx=\left(\sum_{l_1=0}^\infty\frac{1}{l_1!}(-A(1))^{l_1}\right)\Gamma(K+1)+O(\Gamma(K+1)\exp(-C_{18}K))\:.\tag{4.1}
\]
From Lemma \ref{insert3} and (4.1) we obtain
\[
\int_0^{1/2}|g(x)|^{K}dx=\left(\sum_{l_1=0}^\infty\frac{1}{l_1!}(-A(1))^{l_1}\right)\Gamma(K+1)+O(\Gamma(K+1)\exp(-C_{19}K))\:.\tag{4.2}
\]
Since 
$$\int_0^{1/2}|g(x)|^{K}dx=\int_{1/2}^{1}|g(x)|^{K}dx\:,$$
this concludes the proof of Theorem \ref{main}.\\

\vspace{5mm}
\noindent\textbf{Acknowledgments.} We are thankful to Goubi Mouloud for the information that $A(1)=\log 2\pi -\gamma$, which is proved in the paper \cite{baez}. 


\vspace{5mm}
\begin{thebibliography}{99}%

\bibitem{baez} L. B\'aez-Duarte, M. Balazard, B. Landreau, E. Saias, \textit{\'Etude 
de l'autocorr\'elation multiplicative de la fonction `partie fractionnaire'} (French), The Ramanujan Journal, 9(2005), 215--240.

\bibitem{balaz1} M. Balazard, B. Martin, \textit{Comportement local moyen de la fonction de Brjuno} (French) [Average local behavior of the Brjuno function], Fund. Math., 218(3)(2012), 193--224.
\bibitem{balaz2} M. Balazard, B. Martin, \textit{Sur l'autocorr\'elation multiplicative de la
fonction``partie fractionnaire" et une fonction d\'efinie par J. R. Wilton}, arXiv: 1305.4395v1.
\bibitem{bettin} S. Bettin, \textit{On the distribution of a cotangent sum}, Int. Math. Res. Notices (2015), doi: 10.1093/imrn/rnv036 

%\bibitem{billi} P. Billingsley, \textit{Probability and Measure}, John Wiley, New York, 1995.
\bibitem{bre} R. de la Bret\`eche and G. Tenenbaum, \textit{S\'eries trigonom\'etriques \`a coefficients arithm\'etiques}, J.  Anal. Math., 92(2004), 1--79.
%\bibitem{harman} G. Harman, \textit{Metric Number Theory}, Oxford Univ. Press, Oxford, New York, 1998.
\bibitem{hens} D. Hensley, \textit{Continued Fractions}, World Scientific Publ. Co., Singapore, 2006.
\bibitem{mr} H. Maier and M. Th. Rassias, \textit{Generalizations of a cotangent sum associated to the Estermann zeta function}, Communications in Contemporary Mathematics, 18(1)(2016), doi: 10.1142/S0219199715500789.

\bibitem{mr2} H. Maier and M. Th. Rassias, \textit{The order of magnitude for moments for certain cotangent sums}, Journal of Mathematical Analysis and Applications, 429(1)(2015), 576--590.

%\bibitem{mrasympt} H. Maier and M. Th. Rassias, \textit{Asymptotics for moments of certain cotangent sums}, arXiv:1606.03131.

\bibitem{Marmi} S. Marmi, P. Moussa, J. -C. Yoccoz, \textit{The Brjuno functions and their regularity properties}, Commun. in Mathematical Physics, 186(1997), 265--293.

\bibitem{rasthesis} M. Th. Rassias, \textit{Analytic investigation of cotangent sums related to the Riemann zeta function}, Doctoral Dissertation, ETH-Z\"urich, Switzerland, 2014.

\end{thebibliography}
\end{document}